\documentclass[a4paper,10pt]{article}

\title{Actions of solvable Baumslag-Solitar groups on surfaces with (pseudo)-Anosov elements}
\author{Juan Alonso, Nancy Guelman and Juliana Xavier}
\date{}

\usepackage{psfrag}
\usepackage{fancyhdr}
\usepackage{mathrsfs}
\usepackage{verbatim}
\usepackage[ansinew]{inputenc}
\usepackage{amsmath,amsthm,amssymb,amscd}

\newcommand{\N}{\mathbb{N}}
\newcommand{\Z}{\mathbb{Z}}

\newcommand{\R}{\mathbb{R}}
\newcommand{\C}{\mathbb{C}}
\newcommand{\D}{\mathbb{D}}
\newcommand{\T}{\mathbb{T}}

\DeclareMathOperator{\fix}{Fix}
\DeclareMathOperator{\per}{Per}
\DeclareMathOperator{\homeo}{Homeo}
\DeclareMathOperator{\diff}{Diff}

\DeclareMathOperator{\id}{Id}
\DeclareMathOperator{\bs}{BS}

\newtheorem{teo}{Theorem}[section]
\newtheorem{cor}[teo]{Corollary}
\newtheorem{lema}[teo]{Lemma}
\newtheorem{prop}[teo]{Proposition}

\theoremstyle{definition}

\theoremstyle{remark}
\makeindex

\begin{document}

\maketitle

\begin{abstract} Let $BS(1,n)= \langle a,b : a b a ^{-1} = b ^n\rangle$ be the solvable Baumslag-Solitar group, where $n \geq  2$. We study representations of $BS(1, n)$ by homeomorphisms of
closed surfaces of genus $g\geq 1$ with (pseudo)-Anosov
elements.  That is, we consider a  closed surface $S$ of genus $g\geq 1$, and homeomorphisms $f, h: S \to S$ such that $h f h^{-1} = f^n$, for some $ n\geq 2$. It is known that $f$ (or some power of $f$) must be homotopic to the identity.  Suppose
that $h$ is (pseudo)-Anosov  with
stretch factor $\lambda >1$.  We show that $\langle f,h \rangle$ is not a faithful representation of $BS(1, n)$ if  $\lambda > n$.  We also show that there are no faithful representations of $BS(1, n)$ by torus homeomorphisms with $h$ an Anosov map and $f$ area preserving (regardless of the value of $\lambda$).

\end{abstract}

\begin{section}{Introduction}

Baumslag-Solitar groups $\bs(m,n)= \langle a,b : a b^m a ^{-1} = b ^n\rangle$, $m,n \in \Z$, were defined by Gilbert Baumslag and Donald Solitar in 1962 to provide examples of
non-Hopfian groups (see \cite{bsp}). These groups
 are examples of two-generator one-relator groups that play an important role in combinatorial group theory, geometric group theory, topology and dynamical systems  (see for example \cite{bv},
 \cite{fmo}). In particular, the groups $\bs(1,n)$, $n \geq  2$, are the simplest examples of infinite non
abelian solvable groups. Actions of solvable groups on one-manifolds have been studied by many authors (see, for example, \cite{plante} and \cite{navas}; specifically for $BS(1,n)$-actions see
\cite{glp1}). Results in dimension two appear in \cite{glp}; in particular  several examples of
$BS(1,n)$-actions on $\T ^2$ are exhibited.


The groups $\bs(1,n)$ also provide examples of distortion elements, which are related to Zimmer's conjecture and dynamical aspects in general (see \cite{fh} and \cite{zimmer}).
In particular, J. Franks and M.Handel proved that on a surface $S$ of genus greater than one, any distortion element in the group $\diff^1_0(S,area)$  is a torsion element, and therefore there
are not faithful representations of $\bs(1,n)$ in $\diff^1_0(S,area)$ (\cite{fh} ).  It is then natural to wonder what dynamics are allowed for Baumslag-Solitar group actions.

Throughout this paper $S$ will be a closed connected surface of genus $g\geq 1$, and $f, h: S \to S$ homeomorphisms satisfying the Baumslag-Solitar ($\bs$) equation $$h f h^{-1} = f^n,$$\noindent for some
$ n\geq 2$.  We will also assume that $f$ is isotopic to the identity, which  yields no loss of generality due to the following result:

\begin{prop} \label{gl} Let $ \langle f, h \rangle$ be an action of $\bs(1, n)$ on a closed surface $S $. If $n\geq 2$, there exists $k\geq 1$ such that $f^k$ is isotopic to the identity. Moreover, if the action is faithful on $\T^2$,  $f^
k$  has a lift whose rotation set
is the single point
$\{(0,0)\}$ and the set of $f^k$- fixed points is non-empty.
\end{prop}

This is  consequence of Theorem 1.2 of \cite{flm} that claims that every element of infinite order in the mapping class group has linear growth. So, there are no distortion elements of
infinite order in the mapping class group. Since $f$ is a distortion element of $BS(1,n)$, it follows that $[f]$ has finite order. The statement on faithful actions on $\T^2$ is the content of
Theorem 3 in \cite{glp}.

\bigskip

If the action is faithful, the group $\langle f^k, h \rangle$ is isomorphic to $BS(1, n)$.  Then, the previous theorem allows us to restrict
our study to the case where $f$ is isotopic to identity.  Moreover, if $S= \T ^2$ we may assume that $f$ has a lift $\tilde f$ to the universal
covering space such that the rotation set is $\{(0,0)\}$ and that $\fix(f)\neq \emptyset$.  We say that $\tilde f$ is the {\it irrotational lift} for $f$.\\

Once the homotopy class of $f$ has been stablished, we seek to understand the possibilities for the homotopy class of the conjugating generator $h$.  We  exhibit examples of faithful
 actions $ \langle f, h \rangle$ of $BS(1, n)$ where $h$ is isotopic to a (pseudo)-Anosov map (see Section \ref{ex}). However, these examples are somewhat trivial, in the sense that the action
 is supported in an invariant subset which is collapsed by a semiconjugacy taking $h$ to its pseudo-Anosov representative (the construction is detailed in Section \ref{ex}).  This led us
  to consider the case where $h$ is a  (pseudo)-Anosov homeomorphism (and not just isotopic to one). It is known that the centralizer of a (pseudo)-Anosov map is virtually cyclic and so there are
  no faithful actions of
$\Z  ^2$ on surfaces with pseudo-Anosov elements (see \cite{mccarthy}, \cite{rocha} and also \cite{py}). The present work is a natural generalization of this kind of results.  In the more general
scope, our work lies in the context of trying
to understand the nature of the obstructions to the existence of (faithful)
group actions on certain phase spaces (a survey of these ideas can be found in  \cite{fisher}).  Our first result is the following:

\begin{teo}\label{teo1} Let  $ \langle f, h \rangle$ be an action of $BS(1, n)$ on a closed surface $S$, where $f$ is isotopic to the identity, and $h$ is a  (pseudo)-Anosov homeomorphism with stretch factor $\lambda > n$. Then  $f= \id$.

\end{teo}

It follows that there are no faithful representations of $BS(1, n)$ by surface homeomorphisms with $h$ a (pseudo)- Anosov map with stretch factor $\lambda > n$.

As was already pointed out, the last theorem is false if we  ask only that $h$ is {\it isotopic} to a (pseudo)-Anosov map as the examples in Section \ref{ex} show.  Moreover, an example of
G. Mess \cite{fh} shows that the theorem is no longer true with a slight generalization on the acting group: if we set $h$ to be a linear Anosov map on   $\T^2$
with eigenvalue $\lambda$ and
$f(x) = x+w$, where $w \neq 0$ is an  unstable eigenvector of $h$, then $hfh^{-1} (x) = x +\lambda w$.

When the surface is a torus and $f$ is area preserving, we are able to remove the hypothesis $\lambda > n$:

\begin{teo}\label{toro} Let  $ \langle f, h \rangle$ be an action of $BS(1, n)$ on  $\T^2$, where $f$ is an area-preserving homeomorphism isotopic to the identity and $h$ an Anosov homeomorphism. Then, $f= id$.

\end{teo}

So, there are no faithful representations of $BS(1,n)$ by torus homeomorphisms with $h$ an Anosov map and $f$ area preserving (regardless of the value of $\lambda$).  This result
is a nice application of the work of A. Koropecki and F. Tal in \cite{korotal}.

The idea behind these results is making the structures of identity isotopies and of (pseudo)-Anosov maps interact.  Namely, we use the transverse measures of the foliations of the (pseudo)-Anosov
to measure the trajectories $(f_t (x))_{t\in [0,1]}$ of points defined by the identity isotopy $(f_t)_{t\in [0,1]}$.

Finally, we point out that several questions remain unanswered.  Is it possible to extend Theorem \ref{teo1} for  $\lambda < n$? Does there exist a $BS(1,n)$ action  with $h$ isotopic to a (pseudo)-Anosov map that
is not semi-conjugate to a non-faithful one? Is it possible to classify the dynamics of the action according to the isotopy class of  $h$?

\end{section}

\begin{section}{Preliminaries}

\begin{subsection}{Definitions and notations}

For the remainder of the paper $S$ will be a closed  surface, $f: S \to S$ will be a homeomorphism isotopic to the identity and $h: S \to S$ will be a (pseudo)-Anosov homeomorphism (see \ref{pa} for the definition).
If $g:S\to S$ is any map, $\fix (g)$ is the set of fixed points of $g$, $\per(g)$ is the set of periodic points of $g$, and $\per ^p (g)$ is the set of periodic points of $g$ of period $p$.

If $X$ is a topological space, we say that any continuous function $\gamma: [0,1]\to X$ is an arc in $X$. If $\gamma_1, \gamma _2: [0,1]\to X$ are arcs, we note $\gamma _1 \cdot \gamma_2$ the standard concatenation of arcs
($\gamma _1 \cdot \gamma_2: [0,1]\to X$).  If $\gamma _1 \ldots \gamma _n$ are arcs, we define $\prod _{i=0}^n \gamma _i= \gamma_ 1 \cdot \ldots \cdot \gamma _n $.

\end{subsection}

\begin{subsection}{Identity isotopies}

 For $f: S \to S$ as before, fix an isotopy $(f_t)_{t\in [0,1]}$ such that $f_0 = \id$ and $f_1 = f$.  For all $x\in S$ we note $\gamma _x$ the arc $t\to f_t (x), t\in [0,1]$.  More generally, for
$x\in S$ and $m\geq 1$ define $\gamma ^m_x = \prod _{i= 0} ^{m-1}
\gamma _{f^i (x)}$.

If $S$ is hyperbolic, the homotopy class of $\gamma _x$ with fixed endpoints is independent of the isotopy $(f_t)_{t\in [0,1]}$; it is determined by $f$.  In this case, a homeomorphism $f: S \to S$ admits a unique lift to the universal covering that
commutes with every element of the group of covering transformations.  We call this lift the {\it cannonical lift} of $f$; it is the lift determined by lifting the isotopy $(f_t)_{t\in [0,1]}$ starting in the identity.  In fact,
on a compact surface of genus $g\geq 2$, two isotopies joining the identity to a homeomorphism $f$ are homotopic (see \cite{ham}).

If $S = \T ^2$, however, the homotopy class of $\gamma _x$ with fixed endpoints depends on the isotopy $(f_t)_{t\in [0,1]}$. In this case, we will further assume that $f$ has a lift $\tilde f$ whose rotation set is $\{(0,0)\}$, and that the isotopy is such that the lift $(\tilde f_t)_{t\in [0,1]}$ to $\R ^2$ starting at the identity verifies $\tilde f _ 1 = \tilde f$. (This assumption can be made if $\langle f, h\rangle$ is a faithful action of $BS(1,n)$ on $\T ^2$, by Proposition \ref{gl}).

We say that $x\in \fix (f)$ is {\it contractible} if the loop $\gamma _x$ is null-homotopic.  Analogously, if $x\in \per ^p (f)$, we say that $x$ is contracible, if the loop $\gamma  ^p_x$ is null-homotopic.

\end{subsection}

\begin{subsection}{pseudo-Anosov homeomorphisms}\label{pa}

We say that $h: S \to S$ is a {\it pseudo-Anosov} homeomorphism if there is a pair of transverse measured singular foliations $({\cal F}^u, \nu ^u)$ and $({\cal F}^s, \nu ^s)$ on $S$ and a number $\lambda >1$ such that
$$h \cdot ({\cal F}^u, \nu ^u) = ({\cal F}^u, \lambda \nu ^u) \textrm{ and } h \cdot ({\cal F}^s, \nu ^s) = ({\cal F}^s, \lambda ^{-1}\nu ^s).$$\noindent The measured foliations $({\cal F}^u, \nu ^u)$ and
$({\cal F}^s, \nu ^s)$ are called the unstable foliation and the stable foliation, respectively, and the number $\lambda$ is called the {\it stretch factor} of $h$.  We recall that the action of $h$ on $({\cal F}, \nu)$ is given by
$$h\cdot ({\cal F}, \nu) = (h({\cal F}),h_*( \nu)),$$\noindent where $h_*( \nu)(\gamma)$ is defined as $\nu (h^{-1} (\gamma))$ for any arc $\gamma$ transverse to $h ({\cal F})$.  So, in particular,
\begin{equation}\label{eq1} \nu ^s (h(\gamma)) = \lambda \nu ^s (\gamma) \end{equation} \noindent for any arc $\gamma$ transverse to the stable foliation, and

\begin{equation}\label{eq2} \nu ^u (h(\gamma)) = \lambda^{-1} \nu ^u (\gamma), \end{equation}\noindent for any arc $\gamma$ transverse to the unstable foliation.  Note that this corresponds to the statement that $h$ shrinks the leaves of the stable foliation and
stretches the leaves of the unstable foliation.\\

We can extend the measures $\nu^s$ and $\nu^u$ to arcs not necessarily transverse to the foliations. We do so as follows

$$ \nu(\gamma) = \inf \{ \sum_i\nu(\alpha_i): \alpha = \alpha_1\cdots\alpha_k \sim \gamma \mbox{ rel endpoints, and } \alpha_i \mbox{ transverse to } {\cal F} \}$$

where $\nu$ (and ${\cal F}$) stand for either $\nu^s$ or $\nu^u$ (respectively ${\cal F}^s$ or ${\cal F}^u$). Notice that equations \ref{eq1} and \ref{eq2} still hold for a general arc $\gamma$. By construction, $\nu^s$ and $\nu^u$ are invariant under homotopy with fixed endpoints, and they are subadditive, that is: $$\nu(\alpha\cdot\beta) \leq \nu(\alpha) + \nu(\beta). $$

We will use several properties of pseudo-Anosov homeomorphisms that we state in this section.  The reader not familiar with this theory should refer to Chapter Fourteen in \cite{fm}.

\begin{lema}\label{pap} The following holds:
\begin{enumerate}
 \item Let $({\cal F}, \nu)$ be the stable or unstable measured foliation of a pseudo-Anosov homeomorphism on a compact surface $S$.  Then, $\nu (\alpha) > 0$ for every essential simple
 closed curve $\alpha\in S$.
\item Moreover, $\nu(\alpha) = 0$ if and only if $\alpha$ is homotopic with fixed endpoints to a leaf segment of ${\cal F}$.
\item For any leaf segment $l$ of ${\cal F}^u$ and any $\epsilon>0$, there is $m\geq 0$ so that $h^m(l)$ is $\epsilon$-dense in $S$.
\item For any pseudo-Anosov homeomorphism $h$ of a compact surface $S$, the periodic points of $h$ are dense.
\item A pseudo-Anosov homeomorphism $h$ of a compact surface $S$ has a dense orbit.
\item The functions $x\to \nu ^s (\gamma _x)$ and $x\to \nu ^u (\gamma _x)$ are continuous, and thus bounded on $S$.
\end{enumerate}

\end{lema}


\subsection{Orientation covers for foliations}\label{oc}

A measured singular foliation $({\cal F}, \nu)$ on a surface $S$ may not be orientable.  It is easy to see that this is the case if the foliation has a singularity with an odd number of prongs. Moreover, the stable and unstable
foliations for a pseudo-Anosov homeomorphism can fail to be orientable (see Corollary 14.13 in \cite{fm}).  However, if ${\cal F}$ is not
orientable, there exists a connected twofold branched cover $$p: \tilde S \to S,$$\noindent called the {\it orientation cover} of $S$ for ${\cal F}$ such that there is an induced measured foliation $(\tilde{\cal F}, \tilde \nu)$ on
$\tilde S$ that is orientable and such that $p$ maps leaves of $\tilde {\cal F}$ to leaves of ${\cal F}$ and $p_* (\tilde \nu) = \nu$.  The branch points of the cover are exactly the preimages under $p$ of the singularities of
${\cal F}$ with an odd number of prongs.

The construction of the orientation cover is similar to the standard construction of the orientation double cover of a non-orientable manifold. We refer the reader to page 403 on \cite{fm} for details.

\end{subsection}

\begin{subsection}{The Baumslag-Solitar equation}

The following is an easy consequence of the group relation $hfh^{-1}=f^n$ that we will need later on.

\begin{lema}\label{bsp} Let $f,h$ be such that $hfh ^{-1} = f ^n$.  Then:
\begin{enumerate}
 \item For all $m\geq 1$, $h^m f h ^{-m} = f ^{n^m} $;
 \item $h^{-1} (\fix (f))\subset \fix (f)$
\end{enumerate}

\end{lema}

\end{subsection}

\end{section}

\begin{section}{Examples of actions with (pseudo)- Anosov classes}\label{ex}

In this section we provide examples of homeomorphisms $f,h:S\to S$ defining a faithful action of  $BS(1,n)= \langle a,b : a b a ^{-1} = b ^n\rangle$ on a surface $S$, where $h$ is isotopic to a (pseudo)- Anosov map.

Let $S$ be a closed surface of genus $g\geq 1$ and $h_1:S\to S$ a (pseudo)- Anosov homeomorphism with a fixed point $x\in S$.

\noindent {\em Example 1:} Consider the standard action of $BS(1,n)$ on $S^2 = \C\cup\{\infty \}$ generated by the M\"obius transformations $f_0(z) = z+1$ and $h_0(z) = nz$. It is known that this action is faithful. Moreover, every orbit is free except for the global fixed point at $\infty$. Blow up this fixed point to a circle to obtain a faithful action $\hat f_0,\hat h_0 : D \to D$ of $BS(1,n)$ on the disk $D$, fixing every point in the boundary $\partial D$.

On the other hand, blow up $x\in S$ to a circle, obtaining a homeomorphism $\hat h_1: \hat S\to \hat S$ on the surface with boundary $\hat S$, fixing every point in $\partial\hat S$.

Glue $D$ and $\hat S$ by their boundaries, obtaining a surface $S'$ homeomorphic to $S$. Define $f$ and $h$ on $S'$ so that $f|_D = \hat f_0$, $f|_{\hat S} = id_{\hat S}$, $h|_D=\hat h_0$ and $h|_{\hat S} = \hat h_1$. Clearly they agree on the identified boundaries.
The resulting action is faithful, since the action on $D$ is. Identifying $S$ and $S'$, we have that $h$ is isotopic to $h_1$. This example is somewhat trivial because the faithful action only happens on a disk whose complement contains all the homology of $S$. Moreover, it is semi-conjugate to the non-faithful action on $S$ defined by letting $b$ act as the identity and $a$ as $h_1$. The semi-conjugacy is $P:\hat S \to S$ the map that collapses $D$.

\noindent {\em Example 2:} Let $y\in S$ and $x_j = h_1^j(y)$ for $j\in \Z$. Blow up each $x_j$ to a circle, and glue back a disk $D_j$ to obtain a surface $S'$ homeomorphic to $S$ as before. To picture it,imagine that the disk $D_j$ has area $2^{-|j|}$. Define $\hat S = S'-\bigcup D_j$ and $\hat h_1: \hat S \to \hat S$ the blow-up map of $h_1$.

Let $\varphi_j:D\to D_j$ be an identification of each $D_j$ with the disk $D$ in the previous example. Also, let $\hat f_0$ and $\hat h_0$ be as before. Define $f:S'\to S'$ by $f|{\hat S} = id_{\hat S}$ and $f|_{D_j} = \varphi_j \hat f_0 \varphi_j^{-1}$. That is, like $\hat f_0$ on each disk and the identity everywhere else. Define $h:S'\to S'$ by $h|_{\hat S} = \hat h_1$ and $h|_{D_j} = \varphi_{j+1}\hat h_0 \varphi_j^{-1}$.

This action is also faithful, and $h$ is still isotopic to $h_1$ (when regarding $S'$ as $S$). We can pick $y$ to have a dense orbit under $h_1$, and in that case the action is free on an open dense set. The complement of this set, i.e. $\hat S$, still contains all the homology of $S$. And the action is again semi-conjugate to the action by the identity and $h_1$, by collapsing each disk $D_j$.


\end{section}

\begin{section} {Actions of $BS(1,n)$ with (pseudo)- Anosov elements}

Recall that if $S$ is hyperbolic, $f$ has a cannonical lift (the unique lift that commutes with the group of covering transformations). If $S= \T ^2$ by the cannonical lift we mean the irrotational lift (see Proposition \ref{gl}). In
both cases, we denote $\tilde S$ the universal covering space of $S$.

\begin{lema}\label{up}  Let $\tilde f: \tilde S \to \tilde S$ be the cannonical lift of $f$, and $\tilde h:\tilde S\to \tilde S$ be any lift of $h$.  Then, the $\bs$ equation holds for $\tilde f$ and $\tilde h$, that is $$\tilde h \tilde f \tilde h ^{-1} = \tilde f ^n .$$

\end{lema}

\begin{proof}   If $S$ is hyperbolic any two isotopies from the identity to a given homeomorphism $g: S \to S$ are homotopic (see \cite{ham}). Both isotopies $(h f_t h ^{-1} )_{t\in[0,1]}$ and $\prod_{i= 0} ^{n-1} f_t $ join the
identity to $f^{n}$ (the product
stands for concatenation of arcs in the space $\homeo (S)$). In particular,
the arcs $\gamma ^{n}_x$ and $h (\gamma_{h^{-1}(x)})$ are homotopic with fixed endpoints.  This means that $\tilde f ^{n} (\tilde x) = \tilde h  \tilde f \tilde h ^{-1} (\tilde x)$ for all $\tilde x \in \tilde S$ as we wanted.

If $S = \T ^2$, as $\tilde f$ is the irrotational lift for $f$, we have that $f^p (x) = x$ implies $\tilde f ^p (\tilde x) = \tilde x$ for any lift $\tilde x$ of $x$. We know that  $\tilde h \tilde f$ and $\tilde f ^n \tilde h$ are
both lifts of the same map $hf$.  Moreover, by Proposition \ref{gl} $f$ has a fixed point $x$ and $\tilde f (\tilde x) = \tilde x$
for any lift $\tilde x$ of $x$. So, $\tilde h \tilde f \tilde x = \tilde h \tilde x$.  Besides, $f^n h (x) = h f (x) = h (x)$ implies $\tilde f ^n \tilde h (\tilde x) = \tilde h (\tilde x)$.  So, the lifts $\tilde h \tilde f$ and $\tilde f ^n \tilde h$ of $hf$ coincide over $\tilde x$, which implies they are identical, that is $$\tilde h \tilde f
\tilde h ^{-1} = \tilde f ^n$$\end{proof}

\begin{cor}\label{arcsh}  For any $x\in S$ and $m\geq 0$, the arcs $\gamma ^{n^m}_x$ and $h^m (\gamma_{h^{-m}(x)})$ are homotopic with fixed endpoints.
\end{cor}

\begin{proof}  The lemma above implies that for all $m\geq 0$, $\tilde h^m \tilde f \tilde h ^{-m} = \tilde f^{n^m}$, where $\tilde h$ is any lift of $h$ and $\tilde f$ is the cannonical lift for $f$ (see Lemma \ref{bsp} item 1).
The lift $\tilde \gamma ^{n^m} (x)$ of $\gamma ^{n^m} (x)$ starting at $\tilde x$ ends at $\tilde f ^{n^m} (x)$, and the lift $\tilde \gamma _{h^{-m}(x)}$ of $\gamma _{h^{-m}(x)}$ starting at $\tilde h ^{-m} (\tilde x)$ ends at
$\tilde f \tilde h ^{-m} (x)$.  So, $\tilde h^m \tilde \gamma_{h^{-m} (x)}$ has the same endpoints that $\tilde \gamma ^{n^m} (x)$, finishing the proof.
\end{proof}

\begin{cor}\label{arcs}  For any $x\in S$ and $m\geq 0$, the following equations hold: $$\nu ^s (\gamma ^{n^m}_x) = \lambda ^m \nu ^s (\gamma_{h^{-m}(x)}),$$
$$\nu ^u (\gamma ^{n^m}_x) = \lambda ^{-m} \nu ^u (\gamma_{h^{-m}(x)}).$$

\end{cor}

\begin{proof}  We know that $\nu ^s$ and $\nu ^u$ are invariant by homotopy with fixed endpoints.  So, by Corollary \ref{arcsh} we obtain:
$$\nu ^s (\gamma ^{n^m}_x) = \nu ^s (h^m (\gamma_{h^{-m}(x)})) = \lambda ^m \nu ^s (\gamma_{h^{-m}(x)}),$$
$$\nu ^u (\gamma ^{n^m}_x) = \nu ^u (h^m (\gamma_{h^{-m}(x)})) = \lambda ^{-m} \nu ^u (\gamma_{h^{-m}(x)}).$$

\end{proof}

\begin{lema}\label{fixcont} Every $f$- fixed point is contractible.

\end{lema}

\begin{proof} If $x\in \fix (f)$, then $\gamma _x$  is a loop. So, for all $m\geq 0$, $\nu ^u (\gamma _x ^{n^m}) = n^m\nu ^u (\gamma _x)$. We want to show that $\gamma _x$ is homotopically trivial for every $x\in \fix(f)$.  Recall that $\nu^u (h(\gamma)) = \lambda ^{-1} \nu^u (\gamma)$. By Corollary \ref{arcs},
$$\lambda ^{-m}\nu ^u (\gamma_{ h^{-m}(x)}) = \nu ^u (\gamma _x ^{n^m}) = n^m\nu ^u (\gamma _x).$$ We obtain $(\frac{1}{\lambda n})^m \nu ^u(\gamma_{h^ {-m}(x)}) = \nu ^u (\gamma _x)$ for all $m\in \N$.  As
$\nu^u (\gamma_{h^ {-m}(x)}) $ is bounded, this implies that $\nu ^u (\gamma _x) = 0$, and therefore that $\gamma _x$ is a trivial loop by Lemma \ref{pap} item 1.\end{proof}

%

%
%
\end{section}

\begin{lema}\label{nus0}  If $\lambda > n$, then $ \nu ^s (\gamma _x ) = 0$ for all $x\in S$.

\end{lema}

\begin{proof}  By Lemma \ref{arcs}, $$\nu ^s (\gamma ^{n^m}_x) =  \lambda ^m \nu ^s (\gamma_{h^{-m}(x)}) .$$ \noindent So,
 $$\nu ^s (\gamma _{h ^{-m}(x)}) =  \lambda ^{-m} \nu ^s (\gamma_x^{n^m}) \leq \lambda ^{-m} \sum _{i=0}^{n^m-1} \nu ^s (\gamma _{f^i (x)})\leq \left( \frac{n}{\lambda}\right) ^m C , $$\noindent where $C$ is a bound for the function
 $y\to \nu ^s (\gamma _y)$ on $S$. Since this inequality holds for all $x\in S$, we can apply it to $x_m = h^m(y)$ for any $y\in S$, so we get that $$\nu^s (\gamma_y) \leq \left( \frac{n}{\lambda}\right) ^m C  $$ for all $y\in S$ and $m$ a positive integer. Since we are supposing that $\lambda > n$, this implies the lemma.

\end{proof}

\begin{lema}\label{fixFs} If $ \nu ^s (\gamma _x ) = 0$ for all $x\in S$, then $f$ preserves each leaf and semi-leaf of ${\cal F}^s$.

\end{lema}

\begin{proof} That $f$ preserves the stable leaves is a direct consequence of the hypothesis, since any curve with endpoints in different leaves of ${\cal F}^s$ must have positive $\nu^s$ measure. To show that it preserves
semi-leaves, let $x\in S$ be a singularity of ${\cal F}^s$. Consider a lift of the action to the universal cover $\tilde S = \D$ (it exists by Lemma \ref{up}). Let $\tilde f$ and $\tilde h$ be the respective lifts of $f$ and
$h$. Also, let $\tilde x \in \D$ be a preimage of $x$, and $(\tilde {\cal F}^s,\tilde \nu^s)$ be the measured lamination on $\D$ that covers $({\cal F}^s,\nu^s)$. Notice that $\tilde f$ must preserve the leaves of
$\tilde {\cal F}^s$. Now recall that $\tilde {\cal F}^s(\tilde x)$ is homeomorphic to a star of center $\tilde x$, and whose prongs are the semi-leaves of $\tilde {\cal F}^s(\tilde x)$.
Since $\tilde f$ is an homeomorphism preserving $\tilde {\cal F}^s(\tilde x)$, it must fix the center $\tilde x$ and permute the legs.
Moreover, each semi-leaf is a ray from $\tilde x$ that converges to a point in $\partial \D$. Since $f$ is isotopic to the identity, $\tilde f$ extends continuously to $\bar \D$ as the identity on $\partial \D$, and so it
preserves each semi-leaf of $\tilde {\cal F}^s(\tilde x)$. We obtain the lemma by projecting to $S$.

\end{proof}

\begin{cor}\label{fixh} If  $ \nu ^s (\gamma _x ) = 0$ for all $x\in S$, and $x$ is a singularity of ${\cal F }^s$, then $ x\in  \fix (f)$.

\end{cor}

\begin{proof} Direct from Lemma \ref{fixFs}.
\end{proof}



Let $p: \tilde S \to S$ be the orientation cover for ${\cal F}^s$ and $\tilde{\cal F}^s$ the induced oriented measured foliation on $\tilde S$ (see Section \ref{oc}).

\begin{lema} The action lifts to $\tilde S$.  That is,  there exist homeomorphisms $\tilde f, \tilde h : \tilde S \to \tilde S$ such that $p\tilde f = f p$, $p\tilde h = h p$ and $\tilde h \tilde f \tilde h ^{-1}=\tilde f ^n$.  Moreover,
$\tilde f$ is isotopic to the identity and $\tilde h$ is pseudo-Anosov with both stable and unstable foliations orientable.

\end{lema}

\begin{proof} Let $B$ be the set of branch points of $p$, so that $p|_{\tilde S \backslash B}:\tilde S \backslash B\to S\backslash p(B)$
is a covering map.  Then, by the lifting criterion any map $g:S \to S$ preserving $p(B)$ lifts to a map $\tilde g: \tilde S \to \tilde S$ (see Proposition 1.33 in \cite{hatcher}). Moreover, there are two possible lifts for such a map; one preserving the orientation of $\tilde{\cal F}^s$ and
the other one reversing it.  For any such map $g:S \to S$, we will call $\tilde g$ (and $-\tilde g$) to the lift of $g$ preserving (resp. reversing) the orientation of $\tilde{\cal F}^s$.
Recall that $p(B)$ are the singularities of ${\cal F}^s$ with an odd number of prongs. So clearly $h$ preserves $p(B)$.  Moreover, corollary \ref{fixh} tells us that $f$ also preserves $p(B)$.  Notice that $\tilde h \tilde f \tilde h ^{-1}$ and $\tilde f ^n$ are both lifts of the same map $f^n$.  So, either
$\tilde h \tilde f \tilde h ^{-1}=\tilde f ^n$, or $\tilde h \tilde f \tilde h ^{-1}=- \tilde {f ^n}$.  As $\tilde f$ preserves the orientation of $\tilde{\cal F}^s$, so does $\tilde f ^n$, and so $\tilde f ^n$ is the lift of $f^n$
preserving orientation of $\tilde{\cal F}^s$. On the other hand, as both $\tilde h$ and $\tilde f$ are the orientation preserving lifts, $\tilde h \tilde f \tilde h ^{-1}$ preserves the orientation of $\tilde{\cal F}^s$.  So,
$\tilde h \tilde f \tilde h ^{-1}=\tilde f ^n$, as we wanted.

Note that $\tilde f$ is isotopic to the identity because $f$ is.  To see that $\tilde h$ is pseudo-Anosov  with oriented measured foliations, note that the unstable measured foliation $({\cal F}^u, \nu ^u)$ lifts to a singular
measured foliation $(\tilde{\cal F}^u,\tilde \nu ^u)$ on $\tilde S$ that is transverse to $\tilde{\cal F}^s$. So, we can use the orientation on $\tilde{\cal F}^s$ to induce an orientation on $\tilde{\cal F}^u$ by imposing,
for instance, that the stable foliation always crosses the unstable foliation locally from left to right.

\end{proof}

The preceeding lemma allows us to restrict ourselves to the case where both measured foliations of $h$ are oriented. Note that in the case where $S = \T ^2$ and $h$ is an Anosov homeomorphism, this assumption is fulfilled. We devote the next section to proving our results in this setting.

\begin{section}{The orientable case}

Throughout this section, we will assume that the measured foliations of $h$ are globally oriented. We say that an arc $\gamma : I \to S$ is positively transverse to an oriented foliation ${\cal F}$ on $S$ if for any $t_0 \in I$ there exists a neighborhood $U (\gamma (t_0))$ in $S$ and an orientation preserving homeomorphism $g$ between $U$ and an open set $V\subset \R ^2$ sending the foliation ${\cal F}$ onto the vertical foliation, oriented with increasing $y$ coordinate, such that the mapping $t \to p_1 (g(\gamma (t)))$ is strictly increasing in a neighborhood of $t_0$, where $p_1$ denotes projection over the first coordinate.

In this case, we can define signed transverse measures $\hat \nu^s$ and $\hat \nu ^u$, assigning positive measure to positively transverse arcs, and negative measures to negatively transverse arcs. This implies that one has the nice property of additivity, that is $\hat \nu  (\alpha \cdot \beta) = \hat \nu  (\alpha ) + \hat \nu (\beta)$, where $\hat \nu$ is any of the transverse measures.

These signed transverse measures can be extended to any arc in the same fashion as we did for the positive ones $\nu^s$ and $\nu^u$. In this case $\hat \nu(\gamma)$ is invariant under homotopy with fixed endpoints, so no infimum needs to be taken. Moreover, it can be shown that $\hat\nu$ comes from integration of a closed $1$-form on $S$.

Notice that equations \ref{eq1} and \ref{eq2} hold for $\hat\nu^s$ and $\hat\nu^u$, and in both cases $|\hat\nu(\gamma)|\leq \nu(\gamma)$.

We define the functions $\varphi _s , \varphi _u : S \to \R$ by the equations $$\varphi _s (x) = \hat \nu ^s (\gamma_x), \ \varphi _u (x) = \hat \nu ^u (\gamma_x).$$\noindent Notice that $\varphi _s$ and $\varphi _u$ are continuous (thus bounded) by continuity of the stable and unstable foliations of $h$. By Corollary \ref{arcs}, together with additivity, we get that:

\begin{equation}\label{ecu1}\lambda ^m \varphi_s (h^{-m}(x)) = \hat \nu ^s (\gamma _x ^{n^m}) = \sum _{k=0}^{n^m-1} \varphi _s (f^k (x)). \end{equation}

\noindent Analogously,

\begin{equation}\label{ecu2}  \lambda ^{-m} \varphi _u (h^{-m}(x)) = \hat \nu ^u (\gamma _x ^{n^m}) = \sum _{k=0}^{n^m-1} \varphi _u (f^k (x)).\end{equation}

%
%
%
%
%
%
%

\begin{lema}\label{fijos} If $ \nu ^s (\gamma _x ) = 0$ for all $x\in S$, then $x\in \fix (f)$ if and only if $\varphi _u (x) = 0$.
\end{lema}

\begin{proof} By Lemma \ref{fixFs}, $f$ preserves each semi-leaf of ${\cal F}^s$. For $x\in S$, let $\alpha_x$ be a stable semi-leaf segment between $x$ and $f(x)$. Then $\alpha_x$ is transverse to ${\cal F}^u$, and we have that
$|\hat \nu^u(\alpha_x)|=\nu^u(\alpha_x)$ and $\hat \nu^u(\alpha_x)=0$ if and only if $\alpha_x$ is a point, i.e. $x\in\fix(f)$.
If $x\in \fix (f)$, then $x$ is contractible, wich imples $\varphi _u (x) = 0$ as $\gamma_x$ is homotopically trivial (see Lemmas \ref{fixcont} and \ref{pap} item 1).
\end{proof}




\begin{lema}\label{sign} If $ \nu ^s (\gamma _x ) = 0$ for all $x\in S$, then the sign of the real-valued function $\varphi _u$ is constant.

\end{lema}

\begin{proof} Let $U = \varphi _ u ^{-1} ((0, +\infty))$ and $V = \varphi _ u ^{-1} ((-\infty, 0))$. Suppose that $\varphi_u$ changes sign, so that $U$ and $V$ are both non-empty.  Observe they are disjoint open sets with boundaries
$\partial U = \partial V = \fix (f)$ as Lemma \ref{fijos} gives $\fix (f) = \varphi _u ^{-1} (0)$.
Take $x\in U$ and $B= B (x, \epsilon)\subset U$.  Then, for all $m\geq 0$,  $h^m (B)$ contains a leaf segment of the unstable foliation of length $\lambda ^m \epsilon$.  Moreover, for sufficiently big $m$ $h^m (B)$ intersects both
$U$ and $V$ because the $h$- iterates of unstable leaf segments accumulate all over $S$. As $h^m (B)$ is connected, we obtain $h^m (B)\cap \fix (f)\neq \emptyset$. Now, recall that $\fix (f)$ is $h^{-1}$- invariant
(Lemma \ref{bsp} item 2.).  Then, $B\cap \fix (f)\neq \emptyset$, a contradiction, as $B\subset U$.

\end{proof}

We are now ready to prove Theorem \ref{teo1} in the orientable case:

\begin{proof} If $\lambda>n$, Lemma \ref{nus0} gives us $ \nu ^s (\gamma _x ) = 0$ for all $x\in S$. Now, by Lemma \ref{fijos}, $\fix (f) = \varphi _u ^{-1} (0)$. Therefore,  we have to prove $\varphi _u \equiv 0$.  By Lemma \ref{sign}, we may assume $\varphi _u \geq 0$ (the case $\varphi _u \leq 0$ is analogous).  Then,
for all $x\in S$ the series $\sum _{k=0}^\infty \varphi _u (f^k(x))$ is of positive terms
and either converges or diverges. Its limit can be computed as the limit of any subsequence of partial sums.  In particular, by equation \ref{ecu2},$$\sum _{k=0}^\infty \varphi _u (f^k(x)) = \lim _{m\to \infty} \sum _{k=0}^{n^m -1} \varphi _u (f^k(x)) =
\lim_{m\to \infty} \lambda ^{-m} \varphi _ u (h^{-m} (x)) = 0.$$\noindent Then, every term must be zero, that is $\varphi _ u \equiv 0$, proving the theorem.

\end{proof}

\begin{cor}\label{pimba} If $ \nu ^s (\gamma _x ) = 0$ for all $x\in S$, then $f= \id$.

\end{cor}

\begin{proof}  Note that in the previous proof we only needed $ \nu ^s (\gamma _x ) \equiv 0$ to conclude.

\end{proof}

\end{section}

\begin{section} {The conservative case in the torus}

In this section we prove Theorem \ref{toro}.  This result is a consequence of the work of A. Koropecki and F. Tal in \cite{korotal}.

We say that a closed subset $K\subset T ^2$ is {\it fully essential} if any loop in $\T ^2 \backslash K$ is null- homotopic.

\begin{teo}[Theorem A in \cite{korotal}]\label{kt} Let $f: \T ^2 \to \T ^2$ be an irrotational homeomorphism preserving a Borel probability measure of full support, and let $\tilde f$ be its irrotational lift. Then
one of the following holds:

\begin{enumerate}
 \item $\fix (f)$ is fully essential;
 \item Every point in $\R ^2$ has bounded $\tilde f$- orbit;
 \item $\tilde f$ has uniformly bounded displacement in a rational direction; i.e. there is
a nonzero $v\in \Z ^2$ and $M>0$ such that
 $$|<\tilde f ^n (z)-z, v>|\leq M$$\noindent for all $z\in \R ^2$ and $n\in \Z$.
\end{enumerate}

\end{teo}

We assume in this section that $h : \T^2 \to T ^2$ is an Anosov homeomorphism,  $f : \T^2 \to T ^2$ is isotopic to the identity preserving a Borel probability measure of full support, and $hfh ^{-1} = f ^n$ for some $n\geq 2$.

\begin{lema}\label{unacomp} $T^2 \backslash \fix (f)$ is connected and $h^{-1}$ - invariant.

\end{lema}

\begin{proof} By Lemma \ref{bsp} item 2, $T^2 \backslash \fix (f)$ is $h^{-1}$ invariant.  Suppose that $T^2 \backslash \fix (f)$ has two different connected components $U$ and $V$ and take a segment of unstable leaf $\delta\subset U$.  Then, if $m$ is big enough $h^m(\delta)\cap \fix(f)\neq \emptyset$
because $h^m (\delta)$ is connected and intersects both $U$ and $V$.  But then, $\delta \cap \fix(f)\neq \emptyset$, which contradicts that $\delta\subset U\subset T^2 \backslash \fix (f)$.
\end{proof}

\begin{lema}\label{fixess} $T^2 \backslash \fix (f)$ contains an essential simple closed curve.

\end{lema}

\begin{proof}  Let $D = T^2 \backslash \fix (f)$.  Then, the previous lemma implies that $D$ is connected and $h^{-1}$- invariant. By Lemma \ref{pap} item 4., there exists $x\in D\cap \per (h)$. 
Take a segment $\delta \subset W^u_ h (x)\cap D$ through $x$ and a flow box of the unstable foliation $U\subset D$.  Follow the segment of stable leaf from $x$ until it hits $\delta$ again.  Note 
that this whole
segment is contained in $D$ and that is transverse to  ${\cal F}^u$.  We can modify this segment slightly, just inside $U$, to obtain a loop that is transverse 
to ${\cal F}^u$.  As the foliations
of an Anosov toral homeomorphism are non-singular, it follows that this loop is homotopically non trivial. This loop is the desired essential closed curve.

\end{proof}

\begin{lema}\label{racd}  If $\{\tilde f ^{n^m}(z): m\in \Z\}$ is not bounded, then for any nonzero $v\in \Z ^2$ $|<\tilde f ^{n^m} (z)-z, v>|$ is not bounded.

\end{lema}

\begin{proof}  Note that $$\lim _{m\to \infty} \nu ^u (\gamma_x ^{n^m})=   \lim _{m\to \infty} \lambda ^{-m} \nu ^u (\gamma _{h^{-m}(x)}) = 0$$\noindent (see Corollary
\ref{arcs}).  So,
if $\{\tilde f ^{n^m}(z): m\in \Z\}$ is not bounded it has an irrational asymptotical direction (that of the unstable eigenvector of $\tilde h$).

\end{proof}

\begin{lema}\label{bounded} If $\{\tilde f ^{n^m}(z): m\in \Z\}$ is bounded for all $z\in \R ^2$, then $\nu ^s(\gamma _ x) \equiv 0$.

\end{lema}

\begin{proof} Let $x\in \per^m(h)$, and let $ \widetilde {h ^m}$ be the lift of $h^m$ fixing a lift $\tilde x$ of $x$.  Let $\Phi := \widetilde {h ^m}$, so that $\Phi (\tilde x ) = \tilde x$.  Lemma \ref{up} gives us
$\Phi \tilde f \Phi ^{-1} = \tilde f ^{n^m}$.  Note that
$$\Phi^k \tilde f (\tilde x) = \Phi^k \tilde f \Phi^{-k} (\tilde x) = \tilde f^{(n^m)^k}, $$\noindent and that we are assuming that $\{\tilde f ^{{n^m}^k} (\tilde x): k\in \Z\}$ is bounded.  This implies that $\tilde f (x)\in W^s_{\Phi} (\tilde x)$.  Furthermore,
$W^s_{\Phi} (\tilde x)$ projects to $W^s_{h^m} (x) \subset W ^s _h (x)$.  So, $f(x)\in  W^s _h (x) $ and $\nu ^s (\gamma _ x)= 0$ over the set $\per (h)$.  As this set is dense (Lemma \ref{pap} item 4.), one obtains $\nu ^s(\gamma _ x) \equiv 0$.

\end{proof}

We are now ready to prove Theorem \ref{toro}:

\begin{proof} Recall that Proposition \ref{gl} states that $f$ is irrotational.  We may now apply Theorem \ref{kt}.  Lemma \ref{fixess} states that item 1 does not hold.  If item 2 holds, then Lemma \ref{bounded} tells us that
$\nu ^s(\gamma _ x) \equiv 0$.  Then, $f = \id$ by Corollary \ref{pimba}.  If item 3 holds Lemma \ref{racd} tells us that $\{\tilde f ^{n^m}(z): m\in \Z\}$ is bounded for all $z\in \R ^2$, and we are done by re-applying
Lemma \ref{bounded}.

\end{proof}

\end{section}

\end{document}